\documentclass[12pt,a4paper]{article}%
\usepackage{amsmath}%
\usepackage{amsfonts}%
\usepackage{amssymb}%
\usepackage{graphicx}
\usepackage{amsthm}

\begin{document}

\renewcommand*{\proofname}{\bf Proof}
\newtheorem{theorem}{Theorem}
\newtheorem{claim}{Claim}
\newtheorem{corollary}{Corollary}
\newtheorem{lemma}{Lemma}
\theoremstyle{definition}
\newtheorem{definition}{Definition}
\theoremstyle{remark}
\newtheorem{rem}{\bf Remark}

\sloppy
\righthyphenmin=2
\exhyphenpenalty=10000
\binoppenalty=8000
\relpenalty=8000

\def\q#1.{{\bf #1.}}

\title{Some formulas for the number of gluings}
\author{A.\,V.\,Pastor\thanks{This research was supported by grant of the President of Russian Federation NSh-3229.2012.1; Russian Foundation for Basic Research, grant 11-01-00760-a and grant of Russian government (FCP).}
 \and O.\,P.\,Rodionova}
\date{}
\maketitle

\section{Introduction}

\subsection{Notations}

Consider  $K$ disks $D_1,D_2,\ldots,D_K$. Let the boundary circle of~$D_i$ is also denoted by $D_i$.
Let $2N$ points be marked on the boundary circles $D_1$,\dots, $D_K$ such that at least one point is marked on each circle. 
We fix on the each circle~$D_i$ a  counterclockwise orientation. The marked points divide the circle~$D_i$ into several arcs (which do not contain marked points). One of these arcs  is marked with the number~$i$. 

Thus we have $2N$ arcs on $K$ circles. We split these arcs into pairs and glue together correspondent arcs such that two arcs in each pair are oppositely oriented.  We obtain as a result of gluing a compact orientable surface without boundary. This surface can be disconnected. 
Points marked on the circles and arcs of these circles form a graph drawn on the obtained surface.

In what follows we call   disks  $D_1,D_2,\ldots,D_K$ by  \textit{polygons}. Then  marked points are \textit{vertices}  and arcs of our circles by are \textit{edges} of these polygons. Moreover, let  a  disk with exactly $M$  marked points on its boundary circle be called $M$-gon.  We 
allow polygons with one and two vertices.
For each edge we fix  the {\it first} and the {\it last} vertex (in counterclockwise orientation of their polygon).

\begin{definition}
Let a \textit{map} be an ordered pair $(X,G)$, where $G$ is a graph  (possibly with loops and multiple edges), embedded in a compact orientable surface~$X$ without boundary, such that  connected components of  $X \setminus G$,  called {\it faces} of the map, are  	homeomorphic to discs. 

Two maps~$(X,G)$ and $(X',G')$ are \textit{isomorphic}, if there is an orientation preserving homeomorphism $f: X \to X'$, such that $f(G) = G'$. 

Let the \textit{genus} of the map $(X,G)$ be the genus of the surface~$X$.
\end{definition}

All necessary facts on maps (in particular, about connection between maps and permutations) one can find in~\cite{CM}. We recall some notations, that are necessary for  our paper.

\begin{rem}\label{rem:cycle}
Note, that several nonisomorphic maps  can be correspondent to a graph $G$. We  add to a graph~$G$ an additional construction such that the map correspondent to~$G$ would be unique up to isomorphism.

Let's assign to each edge of the graph $G$ a pair of oppositely oriented arcs and for every vertex put in cyclic order all arcs with the beginning at this vertex, (i.e.\ for every vertex let's set a clockwise order in which outgoing arcs pass out this vertex). It is known that fixing a cyclic order of outgoing arcs for every vertex of the graph~$G$  determine up to isomorphism a map correspondent to~$G$  (see, for examples,~\cite{CM}).
\end{rem}

Let's set a cyclic order of outgoing arcs for every vertex of a graph~$G$.  We obtain as a result a permutation $\sigma$ on the set $A$ of all arcs of the graph~$G$ ($\sigma(e)$ is the next (in the cyclic order) arc outgoing from the beginning of an arc $e$).
We also define  permutations $\iota,\tau\in S(A)$ as follows. Let $\iota(e)$ be the arc opposite to~$e$, i.e.\ correspondent to the same edge of $G$ but oppositely oriented. Let $\tau=\sigma\iota$. Since $\iota^2=1$, we also have  $\sigma=\tau\iota$.

We say that {\it an arc $e$ of a map $(X,G)$ belongs to a face~$F$}, if $e$ belongs to the boundary of~$F$ and is oriented in the direction of counterclockwise walk around $F$.  Note, that the arc~$\tau(e)$ also belongs to the  face $F$ and is the next arc for~$e$ in the counterclockwise walk of the boundary of $F$. Hence we call the arc $\tau(e)$ by the {\it next} arc  for~$e$, and the arc $\tau^{-1}(e)$ by the \textit{previous} arc for~$e$. Let's also note that cycles of the permutations~$\sigma$ correspond to vertices of the map $(X,G)$;  cycles of the permutations~$\iota$ correspond to edges of the map $(X,G)$; and cycles of the permutations~$\tau$ correspond to faces of the map $(X,G)$.

It is easy to see that the procedure of gluing  from $K$ polygons gives us a map  with $N$ edges and $K$ faces. Faces of obtained map correspond to polygons we have glued, so let's enumerate these faces with numbers from 1 to $K$ as well as correspondent polygons. Each edge of the map corresponds to two edges of polygons. Counterclockwise orientations of edges of all  polygons gives us a bijection from the set of edges of polygons to the set of arcs of the map.  Let's mark all arcs of the map correspondent to marked edges of polygons with the same number.

Thus we have  obtained a map with $N$ edges and $K$ faces. Faces  of this map  are marked with integers from 1 to $K$, and  each face has exactly one  arc marked with its number.  We say that such map is \textit{marked}.

We say that marked maps $(X,G)$ and $(X',G')$ are {\it isomorphic}, if there is an isomorphism of maps $f:(X,G)\to(X',G')$, that preserves numbers of faces and marks on arcs of maps. 
It is easy to see that the procedure of gluing described above define a map uniquely up to isomorphism. Moreover, every map can be glued from a collection of polygons uniquely up to isomorphism. Speaking about number of marked maps or about number of gluings of some type we  will    always mean the number of maps or gluings up to isomorphism (i.e.\ the number of classes of isomorphic marked maps).

Similarly, speaking about marked maps of some type we mean classes of isomorphic marked maps of this type.

\subsection{Setting of the problem}

The main object of our research is the number of gluings together   $K$ polygons, that have (together)  $2N$ edges,  into a connected orientable surface of genus $g$.

\begin{rem} This problem have an equivalent formulation. Let's enumerate all arcs of the map with integers from  1 to $2N$ in the following way: at first we enumerate arcs of the face 1 beginning at the arc marked with 1, then we similarly enumerate arcs of the face~2, and so on.
Then the permutations $\iota$ and $\tau$ defined above act on the set~$\{1,2,\ldots,2N\}$. 
Moreover,  $\iota$ consists of  $N$ independent transpositions and $\tau$  consists of  $K$ independent cycles, such that elements of each cycle are consecutive integers. It is  known  (see~\cite{CM}), that such a pair of permutations defines a map up to isomorphism. Note, that  cycles of the permutation~$\sigma=\tau\iota$ correspond to vertices of the map that we construct. Since this map has  $V$ vertices, $N$ edges and $K$ faces, the  genus of a surface on which our map is drawn can be calculated by Euler's formula  $V-N+K=2-2g$. 
Thus our problem can be reduced to a problem of finding the number of pairs of permutations $(\iota,\tau)$ of type described above, such that the  group  $\langle\iota,\tau\rangle$ generated by the permutations $\iota$ and $\tau$   is transitive and $\tau\iota$ is a composition
of $N-K+2-2g$ independent cycles.
\end{rem}

\begin{definition}
We denote by~$\varepsilon_g(N,K)$ the number of ways to glue together $K$ polygons, that have (together)  $2N$ edges,  into  connected orientable surface of genus $g$.

For $N=0$ we set  $\varepsilon_0(0,1)=1$ and $\varepsilon_g(0,K)=0$ if $g+K>1$.
\end{definition}

The number $\varepsilon_g(N,K)$ can be interpreted as the number of marked  maps with $N$ edges and  $K$ faces on a surface of genus~$g$. In this interpretation the number $\varepsilon_0(0,1)$  corresponds to a sphere with single marked vertex.

\begin{rem}
\label{rem:zero}
As it was  mentioned above, the number of vertices of the map is equal to $N-K+2-2g$. Since this number must be positive, for
$N<K+2g-1$ we have $\varepsilon_g(N,K)=0$.
\end{rem}

\bigskip
Let's consider one more object: {\it bicolored  gluings}, defined as follows. Let vertices of each of polygons~$D_1,D_2,\ldots,D_K$ be properly colored  with white and black colors and we can glue together only vertices of the same color. The result of such gluing is a  \textit{bicolored  map} which vertices are properly colored  with white and black colors. The edges of polygons and correspondent arcs of the map, outgoing from white (black) vertices we call {\it white} ({\it black}) respectively. Marked  edge in every polygon must be white. 

Note, that bicolored polygon must have even number of vertices and we glue together edges of different colors.

\begin{definition}
Let \textit{marked bicolored map} be a marked map, which vertices are properly colored with white and black colors, such that all marked arcs are white. 

A bicolored map which marked arcs are not necessarily white (i.e.\ can be white or black) is called \textit{quasimarked}.
\end{definition}

\begin{definition}
We denote by $B_g(N,K)$  the number of ways to glue together $K$ bicolored polygons, that have (together)  $2N$ edges,  into  connected orientable surface of genus $g$.

For $N=0$ we set  $B_0(0,1)=1$ and $B_g(0,K)=0$ if $g+K>1$.
\end{definition}

The number $B_g(N,K)$ can be interpreted as the number of marked bicolored  maps with $N$ edges and  $K$ faces on a surface of genus~$g$.

\begin{rem}
\label{r2}
Note, that there is a natural bijection from  the set of quasimarked bicolored maps where the arc marked with the number~$i$ is black to the set of quasimarked bicolored maps where this arc is white: we mark with~$i$ the white arc that is the next to the  marked black arc. The inverse bijection is similar.

Thus, the number of quasimarked maps with $N$ edges and $K$ faces is equal to~$2^KB_g(N,K)$.
\end{rem}

Gluings a surface of genus~$g$ from one polygon were for the first time considered by J.\,Harer and D.\,Zagier~\cite{HZ}.  In this paper the numbers of such gluings (denoted by $\varepsilon_g(N)$) were used to  calculate Euler's characteristic of moduli space. The following recursion for the numbers $\varepsilon_g(N)$ was proved in~\cite{HZ}:
\begin{equation}
  \varepsilon_g(N)=\frac{2N-1}{N+1}(2\varepsilon_g(N-1)+(N-1)(2N-3)\varepsilon_{g-1}(N-2)).
  \label{eq:eps}
\end{equation} 
(We set  $\varepsilon_0(0)=1$.)
Many proofs of formula~\eqref{eq:eps} are known now. In the paper~\cite{GN} a bijective proof for this formula was given.

The recursion for the numbers of bicolored gluings of a surface of genus $g$ from one polygon (denoted by $B_g(N)$) is similar to~\eqref{eq:eps}. It was obtained in the papers~\cite{Jac} and~\cite{Adr} independently:
\begin{equation}
B_g(N)=\frac{1}{N+1}(2(2N-1)B_g(N-1)+(N-2)(N-1)^2 B_{g-1}(N-2)).
\label{eq:B}
\end{equation}                       
(We set $B_0(0)=1$.)  A bijective proof of the formula~\eqref{eq:B} was obtained in~\cite{SV}.

One can find explicit formulas for the numbers~$\varepsilon_g(N)$ and~$B_g(N)$ for small~$g$ in the papers~\cite{GSh} and~\cite{Adr}. We write down two explicit formulas we need in what follows:

\begin{equation}
\varepsilon_0(N)=B_0(N)=\frac{(2N)!}{N!(N+1)!};
\label{eq:eps_B_0}
\end{equation} 

\begin{equation}
\varepsilon_1(N)=\frac{(2N)!}{12N!(N-2)!}.
\label{eq:eps_1}
\end{equation} 

Gluings of a surface of genus~$g$ from two polygons were considered in the papers~\cite{GS} and~\cite{APRW}.
In~\cite{GS} the generating function for the  numbers  of gluings of a surface of genus~$g$ from a $p$-gon and a $q$-gon was obtained, but it was tedious. However, no recurrence or explicit formula for $\varepsilon_g(N,2)$ was obtained in this paper. A more simple  formula was presented in~\cite{APRW}: it was proved that for every $g\ge0$
   $$
   \sum_{N\ge0}\varepsilon_g(N,2)z^N=\frac{P_g^{[2]}(z)}{(1-4z)^{3g+2}},
   $$
where $P_g^{[2]}(z)$ is a  polynomial of degree at most~$3g+1$. A recursive method for calculating this polynomial was presented. Moreover, explicit formulas for the numbers $\varepsilon_0(N,2)$, $\varepsilon_1(N,2)$ and $\varepsilon_2(N,2)$ were proved in~\cite{APRW}:

\begin{equation}
\varepsilon_0(N,2)=N4^{N-1};
\label{g0}
\end{equation} 

\begin{equation}
\varepsilon_1(N,2)={1\over12}(13N+3)N(N-1)(N-2)4^{N-3};
\label{g1}
\end{equation} 

\begin{equation}
\varepsilon_2(N,2)={1\over180}(445N^2-401N-210)N(N-1)(N-2)(N-3)(N-4)4^{N-6}.
\label{g2}
\end{equation}

In this paper we give elementary proofs of formulas~\eqref{g0} and~\eqref{g1}. We  obtain explicit formulas for the numbers~$B_0(N,2)$ and~$\varepsilon _0(N,3)$ (see theorems~\ref{th:B_0_2_form} and~\ref{th:eps_0_3}).

\subsection{An operation of deleting an edge}

All formulas proved in our paper are proved by induction. The plan of these proofs is as follows.
We consider an arbitrary marked connected map of genus~$g$  with  $N>1$ edges and $K$ faces. After  deleting  the edge marked with 1 from this map  we obtain a connected map or a pair of connected maps with less number of edges. This help us to obtain a recurrence  equation for the numbers $\varepsilon_g(N,K)$. After that we by induction prove the explicit formula.

Consider the {\it operation of deleting an edge} in details. Let $(X,G)$ be a marked map,  where $X$ is a connected surface of genus~$g$ and $G$ is a graph with $V$ vertices and $N>1$ edges. Let this map be glued from polygons~$D_1,D_2,\ldots,D_K$, with $m_1,m_2,\ldots,m_K$ edges respectively.  Then  $m_1+m_2+\ldots+m_K=2N$. Denote by~$e_i$ the arc marked with~$i$, by~$e_i'$ the arc opposite to~$e_i$ and by~$\tilde e_i$ the edge of the graph~$G$ correspondent to these two arcs. 
Let's delete from~$G$ the edge~$\tilde e_1$. For every vertex of obtained graph~$G_1$ we put in cyclic order all outgoing arcs as in the map~$(X,G)$ (we simply exclude all deleted arcs from the cyclic order). As it was mentioned in remark~\ref{rem:cycle}, there exists a unique  up to isomorphism map  $(X_1,G_1)$ correspondent to the graph~$G_1$  with fixed cyclic order of arcs. If the surfaces $X$ and $X_1$ are homeomorphic then we assume without loss of generality that $X_1=X$. However, these surfaces can be not homeomorphic: for example, if the graph~$G_1$ is disconnected, then, clearly, the surface $X_1$ is disconnected too.

To make the map $(X_1,G_1)$ marked, we must enumerate the faces and mark some  arcs. Consider the following 4 cases:
\begin{enumerate}
 \item \textit{the arcs  $e_1$ and $e_1'$ belong to different faces};
 \item \textit{$e_1$ and $e_1'$ are successive arcs of the face number~$1$};
 \item \textit{ $e_1$ and $e_1'$ are unsuccessive arcs of the face number~$1$,  the graph~$G_1$ is connected};
 \item \textit{$e_1$ and $e_1'$ are unsuccessive arcs of the face number~$1$,  the graph~$G_1$ is disconnected}.
\end{enumerate}
For each of these cases we describe the obtained map or pair of maps, and count the number of ways to obtain each map or pair of maps as a result of applying the operation to all maps satisfying the conditions of the case we consider.

Let's begin the case analysis.

\smallskip \goodbreak
\q1. \textit{The arcs  $e_1$ and $e_1'$ belong to different faces.} 

Let~$e_1'$  belongs to the face number~$j\ne1$. Then after deleting the edge $\tilde e_1$ the faces~1 and~$j$ become glued together. Hence we obtain a map $(X,G_1)$ of genus $g$ with~$N-1$ edges and $K-1$ faces. 
To make this map marked we assign the number~1 to the face appeared as a result of gluing together faces~1 and~j. Then we enumerate all other faces with numbers from 2 to $K-1$ in  increasing order of their numbers in the map~$(X,G)$ and correct the correspondent marks on arcs. It remains to mark with~1 an  arc  of the map $(X,G_1)$. 
If $e_j\ne e_1'$, then we mark with~1 the arc~$e_j$. If $e_j=e_1'$ and $m_1>1$, then we mark with~1 the arc, next to~$e_1$ in the map~$(X,G)$. 
Otherwise~$e_j=e_1'$ and $m_1=1$, then we mark with~1 the arc next to~$e_j$ in the map~$(X,G)$ (see figure~\ref{ris1}). Note, that in the last case we have~$m_j>1$, since the  map $(X,G)$ is connected and~$N>1$. We obtain as a result a connected marked map of genus  $g$ with  $N-1$ edges and $K-1$ faces.  Note, that the face number~1 of this map contains~$m_1+m_j-2$ arcs.

\noindent
\begin{figure}[!hb]
\centerline{\includegraphics[width=.96\columnwidth, keepaspectratio]{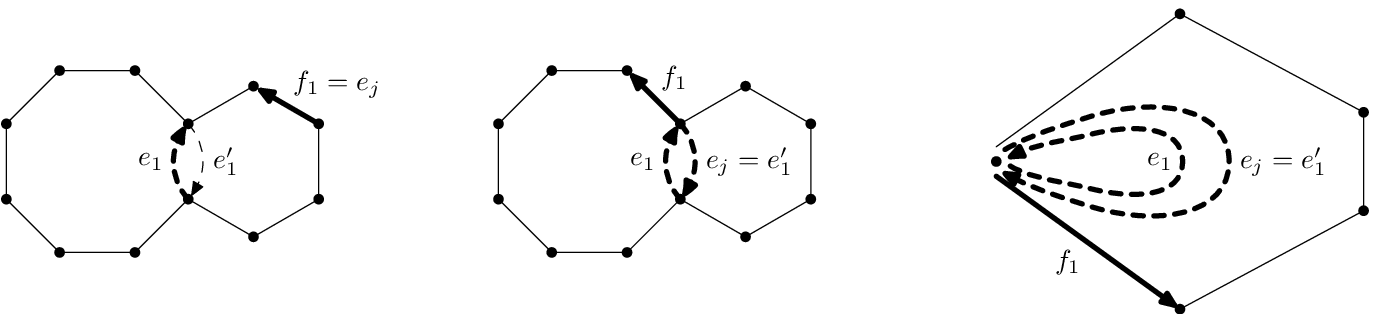}}
\caption{Operation of deleting an edge, case \textbf{1}. In each subcase $f_1$ is the arc that will be marked with~1 in the obtained map.}
\label{ris1}
\end{figure}

\begin{lemma}
\label{l1_1} Let's apply the operation of deleting an edge to all marked maps satisfying the conditions of case~\textbf{\em 1}. 
Then every marked map of genus~$g$ with~$N-1$ edges  and $K-1$ faces, such that face number~$1$ of this map contains~$M$ arcs, is obtained  exactly  $\frac{(M+1)(M+2)(K-1)}{2}$ times.
\end{lemma}

\begin{proof}
Let $(X,G_1)$ be a marked map of genus $g$ with $N-1$ edges and $K-1$ faces. Let $F_1$ be the face number~1 and $D_1$ be the $M$-gon, correspondent to the face~$F_1$. Denote by $f_1$ the arc marked with~1 and by $w$ the beginning of the marked edge of~$D_1$. 
This map was obtained by deleting an edge~$\tilde e_1$ of an initial map, this edge was drawn inside the face~$F_1$. Hence the edge~$\tilde e_1$ is uniquely up to isomorphism defined by an unordered pair $\{a,b\}$ of vertices of the polygon~$D_1$ (the vertices in this pair can coincide). The number of such pairs is equal to~$\frac{M(M+1)}{2}$.

Let's reconstruct  marks on edges of the initial map. At first note, that the face~$F_1$ can be glued from the faces number~1 and number~$j$ of the initial map, where the number~$j$ can be chosen in $K-1$ ways. When we choose some~$j$, we must to point out which part of $F_1$ will have number~1 and which one will have number~$j$. Then we must choose marked arcs in two new faces (the numbers of all other faces and their marked arcs are reconstructed uniquely). Note, that it's enough to choose the arc $e_j$: then we set the enumeration of parts uniquely and the added arc lying in the face number~1 will be marked with~1.

There are two possible cases: $e_j=f_1$ or $e_j=e_1'$ (see the description of the operation). The first case is possible for any edge~$\tilde e_1$, i.e.\ occurs  $\frac{M(M+1)}{2}$ times. The second case is possible only if one of the vertices~$a$ and~$b$ coincides with~$w$. (If exactly one of the vertices~$a,b$ coincides with $w$, then the direction of~$e_j$ is defined uniquely. If both vertices $a$ and $b$ coincide with $w$ then the direction of $e_j$ can be set in two ways.) Thus, the second case occurs~$M+1$ times. Hence, for each~$j$ we have $\frac{(M+1)(M+2)}{2}$ ways, and totally $\frac{(M+1)(M+2)(K-1)}{2}$ ways.
\end{proof}

\smallskip
\q2. \textit{$e_1$ and $e_1'$ are successive arcs of the face number~$1$.}

Since $(X,G)$ is connected and~$N>1$, in this case we have~$m_1>2$, and, consequently, exactly one end of the edge~$\tilde e_1$ has degree~1. 
Denote this end by~$u$. Then after deleting the edge~$\tilde e_1$ and the vertex~$u$ we obtain a map~$(X,G_1')$ with $N-1$ edges and $K$ faces. We preserve the enumeration of faces and marks on all arcs with exception of the deleted arc~$e_1$. Then we mark with~1  the arc of $(X,G)$ which is the next to those of arcs $e_1$ and $e_1'$, that  begins at~$u$ (see figure~\ref{ris2}). As a result we obtain a marked connected map of genus~$g$  with $V-1$ vertices and $K$ faces.

\noindent
\begin{figure}[!hb]
\centerline{\includegraphics{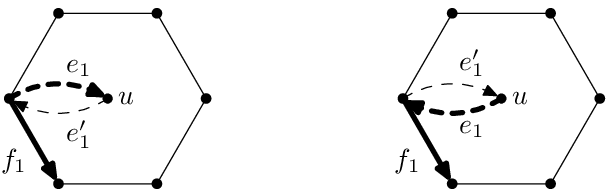}}
\caption{The operation of deleting an edge, case \textbf{2}. Disposition of the arcs~$e_1$ and~$e_1'$.}
\label{ris2}
\end{figure}

\begin{lemma}
\label{l1_2}
 Let's apply the operation of deleting an edge to all marked maps satisfying the conditions of case~\textbf{\em 2}. 
Then every marked map of genus~$g$ with~$N-1$ edges  and $K$ faces is obtained  exactly twice.
\end{lemma}

\begin{proof}
Let $(X,G_1')$ be the marked map of genus $g$ with $N-1$ edges and  $K$ faces; $F_1$ be the face number 1 in this map and $v$ be the beginning of the arc~$f_1$ marked with~1 in this map.  This map can be obtained as a result of deleting an edge~$uv$, such that the vertex~$u$ is inside the face~$F_1$ and the arc $f_1$ was previous arc to~$uv$ before deleting this edge. The edge~$uv$ can be drawn uniquely up to isomorphism, all numbers of faces and arcs marked with numbers more than 1 are preserved. The only non-uniqueness is as follows: we can mark with~1 either $uv$ or~$vu$. Hence we have two variants.
\end{proof}

\smallskip
\q3. \textit{$e_1$ and $e_1'$ are unsuccessive arcs of the face number~$1$,  the graph~$G_1$ is connected.} 

Consider the permutations $\sigma$, $\iota$, $\tau$ on the set of arcs of the map~$(X,G)$ defined above. (The permutation $\sigma$ defines the cyclic  order of outgoing arcs for every vertex; $\iota$ for any arc gives  the opposite arc; $\tau=\sigma\iota$ for any arc gives  the next one).
Let's consider the similar permutations $\sigma_1$, $\iota_1$, $\tau_1=\sigma_1\iota_1$ on the set of arcs of the map~$(X_1,G_1)$. As it was mentioned above, the permutation~$\sigma_1$ is the result of excluding the arcs $e_1$ and $e_1'$ from $\sigma$ and  $\iota_1$  is the result of excluding the cycle $(e_1,e_1')$ from $\iota$. Whence it follows, that
   $$\tau_1(e)=\left\{
     \begin{array}{lr}
      \tau(e),&\tau(e)\not\in\{e_1,e_1'\};\\
      \tau(e_1'),&\tau(e)=e_1;\\
      \tau(e_1),&\tau(e)=e_1'.
     \end{array}
     \right.
    $$
Thus in this case the face 1 of the map $(X,G)$ corresponds to two faces of the map $(X_1,G_1)$.
One of these faces consists of all arcs, lying on the way  from~$e_1$ to~$e_1'$ in the face 1 of  $(X,G)$, and the other face consists of all arcs, lying on the way  from~$e'_1$ to~$e_1$  of the face 1 of  $(X,G)$.  Let's enumerate the first face with~1 and the  second face with~$K+1$ and mark in these faces the arcs, next to~$e_1$ and $e_1'$ respectively (see figure~\ref{ris3}a). We obtain a connected marked map with $N-1$ edges and $K+1$ faces. Since this map~$(X_1,G_1)$ contains~$V$ vertices, by Euler's formula it has genus~$g-1$.

\noindent
\begin{figure}[!hb]
\centerline{\includegraphics[width=.96\columnwidth, keepaspectratio]{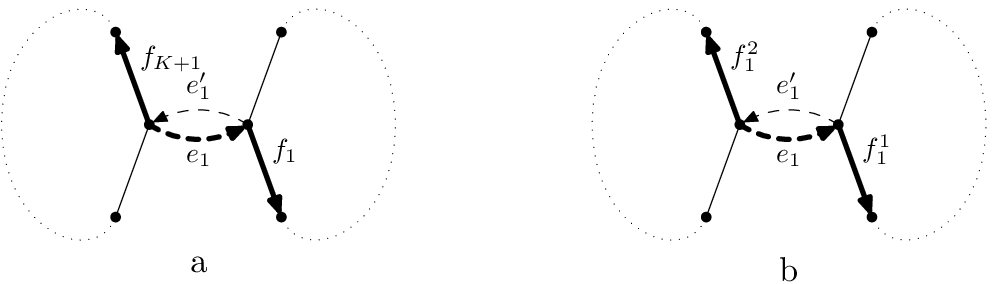}}
\caption{The operation of deleting an edge, cases \textbf{3} and \textbf{4}. In case \textbf{3} (fig.~\ref{ris3}a) $f_1$ and $f_{K+1}$  are the arcs we mark with~$1$ and $K+1$, respectively. In case \textbf{4} (fig.~\ref{ris3}b) $f_1^1$ and $f_{2}^1$  are the arcs we mark with~$1$ in the maps $(X_1^1,G_1^1)$ and $(X_1^2,G_1^2)$,  respectively.}
\label{ris3}
\end{figure}

\begin{rem}
In this case we delete an edge drawn on one of handles of the surface~$X$. After that one connected component of the set~$X \setminus G_1$ becomes homeomorphic to a cylinder. We cut off this cylinder and replace it by two disks correspondent to its bases. As a result we obtain the map  $(X_1,G_1)$.
\end{rem}

\begin{lemma}
\label{l1_3}
 Let's apply the operation of deleting an edge to all marked maps satisfying the conditions of case~\textbf{\em 3}. 
Then every marked map of genus~$g-1$ with~$N-1$ edges  and $K+1$ faces is obtained  exactly once.
\end{lemma}

\begin{proof}
Let $(X_1,G_1)$ be a marked map of genus $g-1$ with $N-1$ edges and $K+1$ faces. Denote by $f_i$ the arc marked with~$i$ and by~$v_i$ its beginning vertex. Let's show that initial marked map $(X,G)$ can be reconstructed uniquely. Note, that the graph~$G$ is a result of adding to $G_1$ the edge~$v_1v_{K+1}$. The cyclic order of outgoing arcs for each vertex is reconstructed uniquely: for a vertex $v_i$ (where $i\in\{1,K+1\}$) the added arc is disposed just before~$f_i$, for all other vertices the cyclic order is preserved. Thus the map $(X,G)$ is reconstructed uniquely.
The numbers of all faces with exception of~1 and~$K+1$ and their marked arcs are preserved too. The arc marked with~1 is also reconstructed uniquely: it leads from~$v_{K+1}$ to $v_1$.
\end{proof}

\smallskip
\q4. \textit{$e_1$ and $e_1'$ are unsuccessive arcs of the face number~$1$,  the graph~$G_1$ is disconnected.} 

In this case the map~$(X_1,G_1)$ consists of two connected components. As in previous case, the face number 1  of the map $(X,G)$ corresponds to two faces of the map $(X_1,G_1)$.
One of these faces consists of all arcs, lying on the way  from~$e_1$ to~$e_1'$ in the face 1 of  $(X,G)$, and the other face consists of all arcs, lying on the way  from~$e'_1$ to~$e_1$  in the face 1 of  $(X,G)$.  Note, that these two faces belong to different connected components of the map~$(X_1,G_1)$. Denote these components by $(X_1^1,G_1^1)$ and $(X_1^2,G_1^2)$ respectively. In both components we enumerate with~1 the face obtained from the face~1 of~$(X,G)$. As in previous case,  we mark  in these faces the arcs, next to~$e_1$ and~$e_1'$ respectively (see figure~\ref{ris3}b). All other faces of each of maps $(X_1^1,G_1^1)$ and $(X_1^2,G_1^2)$ we enumerate in increasing order of their numbers in the map~$(X,G)$ and correct the correspondent marks on arcs. As a result we obtain an ordered pair of marked maps, which have together $V$ vertices, $N-1$ edges and $K+1$ faces. Each of this two maps contains at least one edge.
Let the map  $(X_1^1,G_1^1)$ have genus~$g_1$, and the map  $(X_1^2,G_1^2)$ have genus~$g_2$. By Euler's formula we have~$g_1+g_2=g$.

\begin{rem}
As in previous case, one of the connected components of the set~$X \setminus G_1$ is homeomorphic to a cylinder. However, in the case we consider after replacing this cylinder by two disks correspondent to its bases, the surface becomes disconnected.
\end{rem}

\begin{lemma}
\label{l1_4}
 Let's apply the operation of deleting an edge to all marked maps satisfying the conditions of case~\textbf{\em 4}. 
Consider ordered  pairs of marked  maps with positive number of edges in each map, such that these maps contain together~$N-1$ edges  and $K+1$ faces and have sum of the genera~$g$. Then every such ordered pair   is obtained   exactly  $C_{K-1}^{K_1-1}=C_{K-1}^{K_2-1}$ times, where~$K_1$ is the number of faces of the first  map and  $K_2$ is the number of faces of the second  map.
\end{lemma}

\begin{proof} The proof is similar to the  proof of lemma~\ref{l1_3}, but  in this case the added edge connects the beginnings of arcs marked with~1 in the maps $(X_1^1,G_1^1)$ and $(X_1^2,G_1^2)$. The difference is that  the enumeration of faces can be reconstructed in different ways.  Let's remind, that  we enumerate the faces in each of the maps $(X_1^1,G_1^1)$ and $(X_1^2,G_1^2)$ in the increasing order of their numbers in the initial map.  Hence the enumeration of faces of the initial map is uniquely defined by the choice of the set of numbers that   faces of~$(X_1^1,G_1^1)$ with numbers more than~1 have in the initial map. This set can be chosen in exactly~$C_{K-1}^{K_1-1}$ ways.
\end{proof}

\subsection{Some  useful formulas}

In this section with the help of techniques of the book~\cite{GKP} we proof some equalities for binomial coefficients and similar formulas, that we need in what follows.

\begin{definition}
For integers $i$ and $k$, such that $i\geq 0$, $k>0$ define  $p_k(i)$ as follows:
\begin{equation}
p_k(i)=\prod_{t=0}^{k-1}{(i+1-t)}.
\label{eq:p_k_i}
\end{equation}
We set $p_0(i)=1$ for every $i\geq 0$.
\end{definition}

It if easy to verify   the following equations:
\begin{align}
\label{eq:pk_pk1}
& p_k(i)=(i+2-k)p_{k-1}(i),\quad i\geq 0, k>0;\\
\label{eq:pki_pk1i1}
& p_k(i)=(i+1)p_{k-1}(i-1),\quad i>0, k>0;
\end{align}

\begin{rem}
The numbers~$p_k(i)$ are often called ``descending degrees'' of~$i+1$ and denoted by $(i+1)^{\underline{k}}$.
\end{rem}

\begin{definition}
For integers $M$, $m$, $k$, such that $M\geq m\geq 0$, $k\leq m$ let 
\begin{equation}
A(M,m,k)=\sum_{j=m}^M{\frac{(2j)!}{4^j j!(j-k)!}}.
\label{eq:A_M_m_k_def}
\end{equation}
For $m>M\ge-1$ we set~$A(M,m,k)=0$.
\end{definition}

\begin{lemma}
Let $M$, $m$, $k$ be  integers, such that $M\geq m\geq 0$, $k\leq m$. Then
\begin{equation}
A(M,m,k)=\frac{2}{2k+1}\left(\frac{(M+1-k)(2M+2)!}{4^{M+1}(M+1)!(M+1-k)!}-\frac{(m-k)(2m)!}{4^m m!(m-k)!}\right).
\label{eq:A_M_m_k_form}
\end{equation}
\end{lemma}

\begin{proof}
Let $a_j(k)=\frac{(2j)!}{4^j j!(j-k)!}$, where $k\le j$. It is easy to see, that
   $$a_{j}(k)=\frac{2}{2k+1}\left((j+1-k)a_{j+1}(k)-(j-k)a_j(k)\right).$$
Let's sum these equalities for $j=m,\ldots,M$. The second term of each summand cancel on the first term of the next summand. Hence we obtain the equality we want to prove.
\end{proof}

\begin{corollary}
For all integers $M\ge0$ 
\begin{equation}
\label{eq:A_0_1}
A(M,0,-1)=\sum_{j=0}^M{\frac{(2j)!}{4^j j!(j+1)!}}=2-\frac{2(M+2)(2M+2)!}{4^{M+1}(M+1)!(M+2)!}.
\end{equation}
\end{corollary}

\begin{corollary}
For all integers $M\ge-1$, $k\ge0$ 
\begin{equation}
\label{eq:A_k_k}
A(M,k,k)=\sum_{j=k}^M{\frac{(2j)!}{4^j j!(j-k)!}}=\frac{2}{2k+1}\cdot\frac{p_k(M)(2M+2)!}{4^{M+1}(M+1)!(M+1)!}.
\end{equation}
\end{corollary}

\begin{proof}
For $M\ge k$ this equality directly follows from the formula~\eqref{eq:A_M_m_k_form}. For $M<k$ all terms of the equality we prove are equal to zero.
\end{proof}

\begin{definition}
For integers $N$ and $k$, such that $N,k\geq 0$ let
\begin{equation}
\label{eq:D_k_N}
D_k(N)=\sum_{i=0}^{N}{p_k(i)\frac{C_{2i+1}^i\cdot C_{2N-2i+1}^{N-i}}{2i+1}}.
\end{equation}
\end{definition}

\begin{lemma}
For  every $N\geq 0$ the following statements hold:
\begin{align}
\label{eq:D_0}
& \textrm{\textit{a)} }\quad
D_0(N)=C_{2N+2}^{N}; \\
\label{eq:D_1}
& \textrm{\textit{b)} }\quad
D_1(N)=2\cdot 4^N-\frac{1}{2}C_{2N+2}^{N+1}; \\
\label{eq:D_2}
& \textrm{\textit{c)} }\quad
D_2(N)=(N+1)\left(4^N-\frac{1}{2}C_{2N+2}^{N+1}\right); \\
\label{eq:D_3}
& \textrm{\textit{d)} }\quad
D_3(N)=N(N+1)\left(3\cdot 4^{N-1}-\frac{1}{2}C_{2N+2}^{N+1}\right); \\
\label{eq:D_4}
& \textrm{\textit{e)} }\quad
D_4(N)=(N-1)N(N+1)\left(10\cdot 4^{N-2}-\frac{1}{2}C_{2N+2}^{N+1}\right).
\end{align}
\end{lemma}

\begin{proof}
\textit{a)} We get use of the following formula for convolution of a sum of binomial coefficients (see, for example, \cite[formula~(5.62)]{GKP}):
\begin{displaymath}
\sum_k{C_{tk+r}^{k}\cdot C_{t(N-k)+s}^{N-k}\cdot \frac{r}{tk+r}}=C_{tN+r+s}^{N}.
\end{displaymath}

For  $r=s=1$, $t=2$, we obtain
$$
D_0(N)=\sum_{k=0}^{N}{\frac{C_{2k+1}^{k}\cdot C_{2N-2k+1}^{N-k}}{2k+1}}=C_{2N+2}^{N}.
$$

Thus, the statement \textit{a)} is proved. Before proofs of the next items let's state some useful  formulas. Let
\begin{equation}
\label{eq:d_i_N}
d_i(N)=(i+1)\frac{C_{2i+1}^{i}\cdot C_{2N-2i+1}^{N-i}}{2i+1}.
\end{equation} 

Clearly,  for $N,i>0$ the following statement holds:
\begin{equation}
\label{eq:diN_di1N1}
d_i(N)=2\cdot\frac{2i-1}{i}\cdot d_{i-1}(N-1).
\end{equation}

It follows from the formula~\eqref{eq:pk_pk1}, that $$(2i+1)p_{k-1}(i)=2p_k(i)+(2k-3)p_{k-1}(i)$$ for $i\ge 0$, $k>0$. 
Then taking into account the formulas~ \eqref{eq:pki_pk1i1}, \eqref{eq:d_i_N} and \eqref{eq:diN_di1N1} we obtain the following chain of equalities:
\begin{align*}
&D_k(N)=\sum_{i=0}^{N}{p_k(i)\frac{C_{2i+1}^i\cdot C_{2N-2i+1}^{N-i}}{2i+1}}=
p_k(0)C_{2N+1}^N+\sum_{i=1}^{N}{\frac{p_k(i)}{i+1}d_i(N)}=\\
& =p_k(0)C_{2N+1}^N+\sum_{i=1}^{N}{p_{k-1}(i-1)\frac{2(2i-1)}{i}d_{i-1}(N-1)}=\\
& =p_k(0)C_{2N+1}^N+2\sum_{j=0}^{N-1}{p_{k-1}(j)\frac{2j+1}{j+1}d_{j}(N-1)}=\\
& =p_k(0)C_{2N+1}^N+4\sum_{j=0}^{N-1}{\frac{p_{k}(j)}{j+1}d_{j}(N-1)}+
2(2k-3)\sum_{j=0}^{N-1}{\frac{p_{k-1}(j)}{j+1}d_{j}(N-1)}.
\end{align*} 

Whence by formulas~\eqref{eq:D_k_N} and \eqref{eq:d_i_N} it follows, that for all $k>0$

\begin{equation}
\label{eq:D_k_N_rec}
D_k(N)=p_k(0)C_{2N+1}^N+4D_k(N-1)+2(2k-3)D_{k-1}(N-1).
\end{equation}

\textit{b)} By formulas~\eqref{eq:D_k_N_rec} and~\eqref{eq:D_0}
\begin{align*}
 D_1(N)&=p_1(0)C_{2N+1}^N+4D_1(N-1)-2D_{0}(N-1)=\\
& =4D_1(N-1)+C_{2N+1}^N-2C_{2N}^{N-1}=4D_1(N-1)+\frac{(2N)!}{N!(N+1)!},
\end{align*}
whence by induction we can exclude the following formula: 
\begin{displaymath}
D_1(N)=4^N D_1(0)+\sum_{i=1}^{N}{4^{N-i}\frac{(2i)!}{i!(i+1)!}}.
\end{displaymath}

Note, that $D_1(0)=1$. By formula~\eqref{eq:A_0_1} we have:
\begin{align*}
 D_1(N)&=4^N+4^N\sum_{i=1}^{N}{\frac{(2i)!}{4^i i!(i+1)!}}=4^N\sum_{i=0}^{N}{\frac{(2i)!}{4^i i!(i+1)!}}=\\
& =4^N\left(2-\frac{2(N+2)(2N+2)!}{4^{N+1}(N+1)!(N+2)!}\right)=
2\cdot 4^N-\frac{1}{2}C_{2N+2}^{N+1}.
\end{align*}
 
Thus the statement \textit{b)} is proved.

Note, that $p_k(0)=0$ and $D_k(0)=0$ for $k>1$ (see formulas~\eqref{eq:p_k_i}, \eqref{eq:D_k_N}). Thus we can simplify the formula~\eqref{eq:D_k_N_rec} and find the following recursion for the numbers~$D_k(N)$ for $k>1$:

\begin{align}
 D_k(N)&=4D_k(N-1)+2(2k-3)D_{k-1}(N-1)=\nonumber \\
& =4^N D_k(0)+2(2k-3)\sum_{i=1}^{N}{4^{N-i}D_{k-1}(i-1)}=\nonumber \\
\label{eq:D_k_N_rec_2}
& =2(2k-3)4^{N}\sum_{i=1}^{N}{\frac{D_{k-1}(i-1)}{4^i}}.
\end{align}

\textit{c)} By formula \eqref{eq:D_k_N_rec_2}, \eqref{eq:D_1} and formula \eqref{eq:A_k_k} for $k=0$, we have: 
\begin{align*}
 D_2(N)&=2\cdot 4^{N}\sum_{i=1}^{N}{\frac{D_{1}(i-1)}{4^i}}=
2\cdot 4^{N}\sum_{i=1}^{N}{\frac{2\cdot 4^{i-1}-\frac{1}{2}C_{2i}^i}{4^i}}=\\
& =4^N\cdot N-4^N\sum_{i=1}^{N}{\frac{C_{2i}^i}{4^i}}=(N+1)4^N-4^N\sum_{i=0}^{N}{\frac{(2i)!}{4^i i!i!}}=\\
& =(N+1)4^N-4^N \frac{2(N+1)(2N+2)!}{4^{N+1}(N+1)!(N+1)!}=(N+1)\left(4^N-\frac{1}{2}C_{2N+2}^{N+1}\right).
\end{align*}
 
\textit{d)} By formulas \eqref{eq:D_k_N_rec_2}, \eqref{eq:D_2} and formula \eqref{eq:A_k_k} for $k=1$ we have: 
\begin{align*}
 D_3(N)&=2\cdot 3\cdot 4^{N}\sum_{i=1}^{N}{\frac{D_{2}(i-1)}{4^i}}=
6\cdot 4^{N}\sum_{i=1}^{N}{\frac{i\cdot \left(4^{i-1}-\frac{1}{2}C_{2i}^i\right)}{4^i}}=\\
& =6\cdot 4^{N-1}\sum_{i=1}^{N}{i}-3\cdot 4^N\sum_{i=1}^{N}{\frac{(2i)!}{4^i i!(i-1)!}}=\\
& =N(N+1)\cdot 3\cdot 4^{N-1}-3\cdot 4^N \frac{2N(N+1)(2N+2)!}{3\cdot 4^{N+1}(N+1)!(N+1)!}=\\
& =N(N+1)\left(3\cdot 4^{N-1}-\frac{1}{2}C_{2N+2}^{N+1}\right).
\end{align*}
 
\textit{e)} By formulas \eqref{eq:D_k_N_rec_2}, \eqref{eq:D_3} and formula \eqref{eq:A_k_k} for $k=2$ we have: 
\begin{align*}
 D_4(N)&=2\cdot 5\cdot 4^{N}\sum_{i=1}^{N}{\frac{D_{3}(i-1)}{4^i}}=
10\cdot 4^{N}\sum_{i=1}^{N}{\frac{(i-1)i\cdot \left(3\cdot 4^{i-2}-\frac{1}{2}C_{2i}^i\right)}{4^i}}=\\
& =30\cdot 4^{N-2}\left(\sum_{i=1}^{N}{i^2}-\sum_{i=1}^{N}{i}\right)-
5\cdot 4^N\sum_{i=2}^{N}{\frac{(2i)!}{4^i i!(i-2)!}}=\\
& =10\cdot 4^{N-2}(N-1)N(N+1)-\frac{5\cdot 4^N \cdot 2(N-1)N(N+1)(2N+2)!}{5\cdot 4^{N+1}(N+1)!(N+1)!}=\\
& =(N-1)N(N+1)\left(10\cdot 4^{N-2}-\frac{1}{2}C_{2N+2}^{N+1}\right).
\end{align*} 
\end{proof}

\section{Gluing together  two polygons into a sphere}

In this   section  we prove recurrence and explicit formulas for the numbers $\varepsilon _0(N,2)$. 
The number $\varepsilon _0(N,2)$ can be interpreted as the number of marked maps on a sphere, that contain $N$ edges and 2 faces.	

\begin{theorem}
\label{th:eps_0_2_rec}
The numbers $\varepsilon_0(N,2)$ for $N\geq 2$ satisfy the following recursion:
\begin{equation}
\varepsilon _0(N,2)=2\sum_{i=0}^{N-2}{\varepsilon_0(i)\varepsilon_0(N-i-1,2)}+N(2N-1)\varepsilon_0(N-1).
\label{eq:eps_0_2_rec}
\end{equation} 
\end{theorem}

\begin{proof}
Let's consider all marked maps on a sphere that contain $N>1$ edges and 2 faces and apply to each of these maps the  operation of deleting an edge. Since  $g=0$,  case \textbf{3}  from the description of this operation is impossible.
Hence we have as a result of operation one of these cases:
\begin{itemize}
 \item a marked map on a sphere with $N-1$ edges and one face;
 \item a marked map  on a sphere with $N-1$ edges and two faces;
 \item an ordered pair of two marked maps on  spheres, that contain together  $N-1$ edges and three faces, where each of these maps contains at least one edge.
\end{itemize}

By lemmas~\ref{l1_1} and~\ref{l1_2} each of~$\varepsilon_0(N-1)$ marked maps on a sphere with $N-1$ edges and one face occurs exactly $N(2N-1)$ times, and each of ${\varepsilon_0(N-1,2)}$ marked maps on a sphere with $N-1$ edges and two faces occurs exactly twice.

Consider  ordered pairs of two marked maps on  spheres, that contain together  $N-1$ edges and three faces, such that in each of them the map with one face contains $i$ edges. Clearly,  the number of such pairs is equal to $2\varepsilon_0(i)\varepsilon_0(N-i-1,2)$, and by lemma~\ref{l1_4} each such pair occurs exactly once.

Hence we have the following formula:
\begin{displaymath}
\varepsilon _0(N,2)=N(2N-1)\varepsilon _0(N-1)+2\varepsilon _0(N-1,2)+2\sum_{i=1}^{N-2}{\varepsilon _0(i)\varepsilon _0(N-i-1,2)}.
\end{displaymath}

To complete the proof of lemma it remains to note that $\varepsilon _0(0)=1$.
\end{proof}

\begin{theorem}
\label{th:eps_0_2_form}
For $N\geq 1$ 
\begin{equation}
\varepsilon _0(N,2)=N\cdot 4^{N-1}.
\label{eq:eps_0_2}
\end{equation}
\end{theorem}

\begin{proof} We proof this statement by induction. The base  for  $N=1$ is obvious. Let's prove the induction step. Let $\varepsilon _0(k,2)=k\cdot 4^{k-1}$ for all $k<N$. We will show that  $\varepsilon _0(N,2)=N\cdot 4^{N-1}$.

By formulas \eqref{eq:eps_0_2_rec},   \eqref{eq:eps_B_0}, \eqref{eq:A_0_1} and by  formula \eqref{eq:A_k_k} for $k=0$ we have the following:
\begin{align*}
& \varepsilon _0(N,2)=2\sum_{i=0}^{N-2}{\varepsilon _0(i)\varepsilon _0(N-i-1,2)}+N(2N-1)\varepsilon _0(N-1)= \\
& = 2\sum_{i=0}^{N-2}{\frac{(2i)!}{i!(i+1)!}(N-i-1)4^{N-i-2}}+N(2N-1)\varepsilon _0(N-1)= \\
& = 2\cdot 4^{N-2}\left(N\sum_{i=0}^{N-2}{\frac{(2i)!}{4^i i!(i+1)!}}-\sum_{i=0}^{N-2}{\frac{(2i)!}{4^i i!i!}}\right)
+N(2N-1)\varepsilon _0(N-1)=\\
& =2^{2N-3}\left(2N-\frac{2N^2(2N-2)!}{4^{N-1}(N-1)!N!}-\frac{2(N-1)(2N-2)!}{4^{N-1}(N-1)!(N-1)!}\right)+ \\
&\phantom{=\ } + N(2N-1)\frac{(2N-2)!}{(N-1)!N!}=N\cdot 4^{N-1}.
\end{align*}
\end{proof}

\section{Gluing together  two bicolored polygons into a sphere}

In this section  we prove recurrence and explicit formulas for the numbers  $B_0(N,2)$.
The number  $B_0(N,2)$  can be interpreted as the number of marked bicolored maps on a sphere, that contain $N$ edges and 2 faces. 

\begin{theorem}
\label{th:B_0_2_rec}
The numbers $B_0(N,2)$ for $N\geq 2$ satisfy the following recursion:
\begin{equation}
B_0(N,2)=2\sum_{i=0}^{N-3}{B_0(i)B_0(N-i-1,2)}+\frac{N(N-1)}{2}B_0(N-1).
\label{eq:B_0_2_rec}
\end{equation} 
\end{theorem}

\begin{proof}
The result of applying an operation of deleting an edge to a bicolored  marked map is either a bicolored  quasimarked map or a ordered pair of bicolored  quasimarked maps. Let as apply this operation  to all marked bicolored maps on a sphere with $N>1$ edges and 2 faces. For each map we have exactly one of the following three cases.

\q1. \textit{Two arcs of the deleted edge~$\tilde e_1$ belong to different faces.} 

\noindent Let's remind that all marked arcs of a marked bicolored map are white. In particular, the arc~$e_1$ is white and the opposite arc~$e_1'$ is black. Hence~$e_1'$ couldn't be marked with the number~2. Consequently, the arc marked with~1 in the obtained map was marked with~2 in the initial map, and this arc is white. Thus in this case we obtain a marked bicolored map on a sphere with~$N-1$ edges and one face. Let's count how many times  such a map occurs.

Let  $(X,G_1)$ be a bicolored marked map on a sphere with $N-1$ edges and one face; $F_1$ be its only face; $D_1$ be the ${(2N-2)}$-gon, 
correspondent to the face~$F_1$ and $f_1$ be the arc marked with~1. This map can be obtained by deleting from the initial map  an edge $\tilde e_1$ drawn inside the face~$F_1$. This edge is defined uniquely up to isomorphism by its ends --- a pair of vertices of different colors of the polygon  $D_1$. The marks are also reconstructed uniquely: the arc $f_1$ was marked with the number~2, and the white arc correspondent to the edge $\tilde e_1$ was marked with~1. 

It remains to count the number of ways to draw the edge~$\tilde e_1$, such that its white arc~$e_1$ and the arc~$f_1$ lie in  different parts of the face $F_1$.  Let's enumerate the vertices of~$D_1$ with integers  from~1  to $2N-2$ clockwise, starting with the beginning of the marked edge of the polygon~$D_1$. Then  white vertices have odd numbers and black vertices have even numbers. 
Let~$a$  and~$b$ be  the beginning and the end of the arc~$e_1$. Then $a$ is odd and~$b$ is even, $a,b\in\{1,\ldots,2N-2\}$. The arcs $e_1$ and $f_1$ lie in different parts of the face $F_1$ if and only if the arcs $e_1'$ and $f_1$ lie in the same part of $F_1$. The latter means that~ $b>a$ (see figure~\ref{ris4}). Now it is easy to see that  the number of such ordered pairs $a,b$ is $\frac{N(N-1)}{2}$. 

Hence each of $B_0(N-1)$ marked bicolored maps on a sphere with~$N-1$ edges and one face occurs exactly~$\frac{N(N-1)}{2}$ times, i.e.\ the case~\textbf{1} occurs for $\frac{N(N-1)}{2}B_0(N-1)$ initial maps.

\noindent
\begin{figure}[!hb]
\centerline{\includegraphics{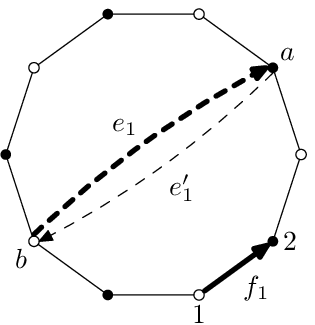}}
\caption{the operation of deleting an edge for a bicolored map, case~\textbf{1}.}
\label{ris4}
\end{figure}

\q2. \textit{Two arcs of the deleted edge are successive arcs of the face number~$1$.} 

\noindent
In this case we obtain a quasimarked bicolored map with~$N-1$ edges and two faces. The arc marked with~2  of this map must be white, the arc marked with~1  can be white or black. As it was shown in remark~\ref{r2} the numbers of such maps with black and white arcs marked with 1 respectively are equal. Hence, the number of maps we consider is~$2B_0(N-1,2)$.

Note, that  every of $2B_0(N-1,2)$  quasimarked maps, satisfying the conditions of case~\textbf{2}, occurs exactly once. Indeed, 
by the same arguments as in the proof of lemma~\ref{l1_2} we have that the deleted edge $uv$ can be reconstructed uniquely up to isomorphism.
Exactly one  of two arcs $uv$ and $vu$ is white, this arc is to be marked with~1. Thus, case \textbf{2} occurs for $\frac{N(N-1)}{2}B_0(N-1)$ initial maps.

\q3. \textit{Two arcs of the deleted edge are unsuccessive arcs of the face number~$1$.}

\noindent Since $g=0$, the graph becomes disconnected, i.e.\ the case we consider corresponds to the case \textbf{4} of the description of the operation of deleting an edge. In this case we obtain an ordered pair of two quasimarked bicolored maps on a sphere,  that contain together~$N-1$ edges and three faces. Each of these faces contains  at least one edge. In these two maps we mark with~1 the arcs, next to  arcs of the deleted edge: in the first map it is the  arc next to  $e_1$  (which is black) and in the second map it is the arc next to~$e_1'$ (which is white). By remark~\ref{r2} we have that the number of quasimarked maps we consider is equal to the number of ordered pairs of marked bicolored maps on a sphere, that contain together $N-1$ edges  and three faces.

By the same reasonings as in the proof of lemma~\ref{l1_4} we have that each ordered pair of marked bicolored maps on the sphere satisfying the conditions of case \textbf{3}   occurs exactly once. Hence,   case~\textbf{3} occurs for the number of initial maps equal to the number of ordered pairs of marked bicolored maps on a sphere, that contain together $N-1$ edges  and three faces.  It is easy to see that the number of such pairs in which the map with one face contains  $i$ edges is equal to $2B_0(i)B_0(N-i-1,2)$.
Since $B_0(1,2)=0$ case \textbf{3} occurs for $2\sum_{i=1}^{N-3}{B_0(i)B_0(N-1-i,2)}$ initial maps.

Since  each of $B_0(N,2)$ initial maps satisfies the condition of one of the cases \textbf{1}, \textbf{2} and \textbf{3}, we obtain the following formula:
\begin{displaymath}
B_0(N,2)=\frac{N(N-1)}{2}B_0(N-1)+2B_0(N-1,2)+2\sum_{i=1}^{N-3}{B_0(i)B_0(N-i-1,2)}.
\end{displaymath} 

Since  $B_0(0)=1$ we have
\begin{displaymath}
B_0(N,2)=2\sum_{i=0}^{N-3}{B_0(i)B_0(N-i-1,2)} + \frac{N(N-1)}{2}B_0(N-1). 
\end{displaymath} 
\end{proof}

\begin{theorem}
\label{th:B_0_2_form}
The numbers $B_0(N,2)$ for $N\geq 1$ satisfy the following  formula:
\begin{equation}
B_0(N,2)=\varepsilon_0(N-1,2)=(N-1)\cdot 4^{N-2}.
\label{eq:B_0_2}
\end{equation}
\end{theorem}

\begin{proof}
We prove by induction, that $B_0(N,2)=\varepsilon_0(N-1,2)$. The base for $N=1$ is obvious. For $N=2$ it follows from formulas~\eqref{eq:B_0_2_rec}, \eqref{eq:eps_B_0} and \eqref{eq:eps_0_2}  that 
$B_0(2,2)=\varepsilon_0(1,2)=1$. 

Let's prove the induction step.
Let the equality $B_0(i,2)=\varepsilon_0(i-1,2)$ holds for all $i\leq N$. We will prove that this equality holds for~$N+1$.  Note, that by formula~\eqref{eq:eps_B_0} for all $i\ge0$ we have $\varepsilon_0(i)=B_0(i)$. 
Moreover, it follows from formula~\eqref{eq:eps_B_0}, that
$$
  \frac{(N+1)N}{2}B_0(N)=\frac{(N+1)N}{2}\varepsilon_0(N)=N(2N-1)\varepsilon_0(N-1).
$$
Then by formulas~\eqref{eq:B_0_2_rec}, \eqref{eq:eps_0_2_rec} and by induction assumption we have, that  
\begin{eqnarray*}
B_0(N+1,2)=2\sum_{i=0}^{N-2}{B_0(i)B_0(N-i,2)} + \frac{(N+1)N}{2}B_0(N)=\\
=2\sum_{i=0}^{N-2}{\varepsilon_0(i)\varepsilon_0(N-i-1,2)} + N(2N-1)\varepsilon_0(N-1)=\\
=\varepsilon_0(N,2).
\end{eqnarray*}
\end{proof}

\begin{rem}
It  would be interesting to find a bijective proof of the identity $B_0(N,2)=\varepsilon_0(N-1,2)$.
Unfortunately we do not know such a proof.
\end{rem}

\section{Gluing together  three polygons into a sphere}

In this section we prove recurrence and explicit formulas for the numbers~$\varepsilon_0(N,3)$. 
Remind, that $\varepsilon_0(N,3)$  is the number of marked maps on a sphere,  that contain $N$ edges and three faces.

\subsection{Recursion}

In this section we need an auxiliary notion.

\begin{definition}
Denote by~$\tilde{\varepsilon}_0(N,M,K)$ the number of marked maps on a sphere with $N$ arcs (i.\,e. $N/2$ edges) and $K$ faces, such that the face number~1 contains exactly $M$ edges.
\end{definition}

Clearly, $\tilde{\varepsilon}_0(N,M,K)$ is also the number of ways to glue a sphere from $K$ polygons, that contain together~$N$ edges where the first polygon contains~$M$ edges.

Note, that  $\tilde{\varepsilon}_0(N,M,K)>0$  if and only if three following conditions hold:
\begin{enumerate}
\item[(i)] $N, M, K \geq 1$; 
\label{cond:1}
\item[(ii)]  $N$ is even;
\label{cond:2}
\item[(iii)] $N\geq M+K-1$.
\label{cond:3}
\end{enumerate}

The condition (iii) follows from the fact that the face number  1  contains  $M$ arcs and each of other faces must have at least one arc.

\begin{lemma}
Let  numbers $N>2$, $K=2$ and $M$ satisfy the conditions {\em (i)-(iii)}. Then the following recursion holds:
\begin{multline}
\tilde{\varepsilon}_0(N,M,2)=2\sum_{i=0}^{\left\lfloor\frac{M-3}{2}\right\rfloor}
{\varepsilon _0(i)\tilde{\varepsilon}_0(N-2i-2,M-2i-2,2)}+\\+(N-M)\varepsilon_0\left(\frac{N}{2}-1\right).
\label{eq:eps_N_M_rec}
\end{multline} 
\end{lemma}

\begin{proof}
Consider all marked maps on a sphere with  $N>2$ arcs ($N/2$ edges) and 2 faces,  such that the face number 1 contains $M$ arcs.
Let's apply to each of this maps the operation of deleting an edge. Consider three following cases.

\noindent
\begin{figure}[!hb]
\centerline{\includegraphics{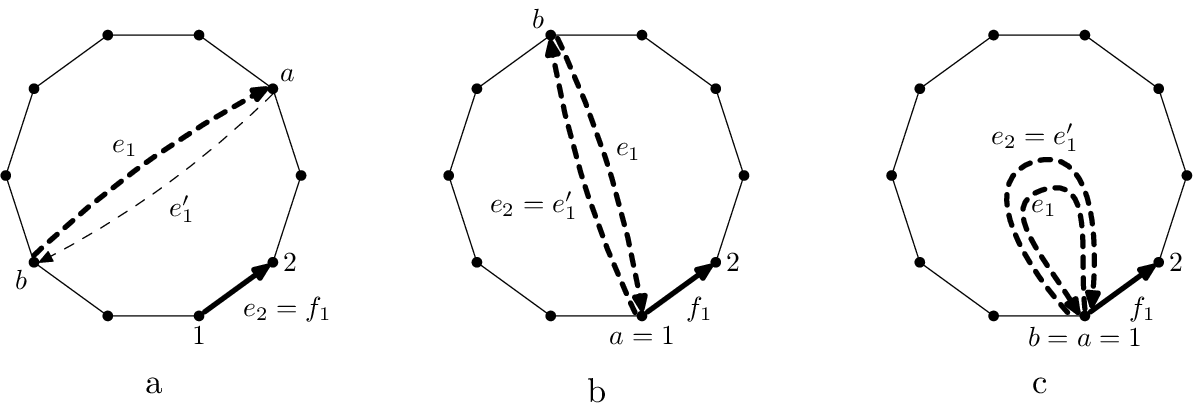}}
\caption{The operation of deleting an edge. Fixed number of arcs of the face~1 of the obtained map.
Case \textbf{1}.}
\label{ris5}
\end{figure}

\q1. \textit{Two arcs of the deleted edge  $\tilde e_1$ belong to different faces.} 

\noindent In this case we obtain a marked map on a sphere with $N/2-1$ edges and one face. Let's count how many times  such a map occurs.

Let $(X,G_1)$ be a marked map with $N/2-1$ edges and one face;  $F_1$ be its only face; $D_1$ be a ${(N-2)}$-gon, correspondent to the face~$F_1$ and $f_1$ be the arc, marked with~1. Similarly to the proof of lemma~\ref{l1_1} let's note, that the initial map (from which  the map $(X,G_1)$ was obtained) is uniquely defined by the deleted edge $\tilde e_1$ (i.e.\ by a pair of vertices of the polygon~$D_1$) and by the arc~$e_2$, marked with~2 in the initial map. 

Let's enumerate vertices of the polygon~$D_1$  with numbers from 1 to $N-2$ clockwise, starting with the beginning of its marked edge. Then the edge~$\tilde e_1$ is defined uniquely up to isomorphism by a pair of numbers
$a,b\in\{1,\ldots,N-2\}$, where $b\ge a$. We consider two subcases. 

\q{1.1}. \textit{$e_2=f_1$.}  In this case any pair~$a,b$ such that
$b-a=M-1$ can be corresponded to the edge $\tilde e_1$ (see figure~\ref{ris5}a). There are $N-M-1$ such pairs.

\q{1.2}. \textit{$e_2$ is an arc of the edge~$\tilde e_1$.} In this case~$a=1$. Let $M>1$. Then $b=N-M$ and the arc marked with 1 in the initial map must precede the arc~$f_1$ (see figure~\ref{ris5}b). These conditions define the arc~$e_2$ uniquely. If~$M=1$, then $b=a=1$ and in the initial map the arc $e_2$ precedes the arc~$f_1$, i.e.\ also is uniquely defined (see figure~\ref{ris5}c).

Thus each of $\varepsilon_0\left(\frac{N}{2}-1\right)$ marked maps on a sphere with  $N/2-1$ edges and  one face occurs $N-M$ times. 
Hence, the case  \textbf{1} occurs for $(N-M)\varepsilon_0\left(\frac{N}{2}-1\right)$ initial maps.

\smallskip
\q2. \textit{Two arcs of the deleted edge are successive arcs of the face number~$1$.} 

\noindent In this case we obtain a marked map on a sphere with  $N/2-1$ edges and two faces, the face number~1
contains $M-2$ arcs. By lemma~\ref{l1_2} each such map is obtained exactly twice. Thus  case \textbf{2} occurs for $2\tilde{\varepsilon}_0(N-2,M-2,2)$ initial maps.

\q3. \textit{Two arcs of the deleted edge are unsuccessive arcs of the face number~1.}

\noindent Since $g=0$, the graph becomes disconnected, i.e.\ the case we consider corresponds to the case \textbf{4} of the description of the operation of deleting an edge. In this case we obtain an ordered pair of two marked  maps on a sphere,  that contain together~$N/2-1$ edges and three faces. Each of these faces contains  at least one edge.  Moreover, the first faces of these two maps contain together $M-2$ arcs. By lemma~\ref{l1_4} each such pair occurs exactly once.  It is easy to see, that the number of such ordered pairs where the map with one face contains~$i$ edges is equal to $2\varepsilon_0(i)\tilde{\varepsilon}_0(N-2i-2,M-2i-2,2)$.
Thus,  the case  \textbf{3}  occurs for $2\sum_{i=1}^{\left\lfloor\frac{M-3}{2}\right\rfloor}
{\varepsilon _0(i)\tilde{\varepsilon}_0(N-2i-2,M-2i-2,2)}$ initial maps.

Summing these  results and taking into account that $\varepsilon _0(0)=1$, we obtain the desired formula.
\end{proof}

\begin{theorem}
\label{th:eps_0_3_rec}
For $N\geq 2$ the numbers $\varepsilon_0(N,3)$ satisfy the following recursion:
\begin{multline}
\varepsilon_0(N,3)=2\sum_{i=0}^{N-2}{\Bigl(\varepsilon_0(i)\varepsilon_0(N-i-1,3)+
\varepsilon_0(i,2)\varepsilon_0(N-i-1,2)\Bigr)}+ \\
+\sum_{i=1}^{2N-3}{(i+1)(i+2)\tilde{\varepsilon}_0(2N-2,i,2)}.
\label{eq:eps_0_3_rec}
\end{multline} 
\end{theorem}

\begin{proof}
Consider all marked maps on a sphere, that contain  $N>1$ edges and  3 faces and apply to each of them the operation of deleting an edge. For any such map we have  exactly one of the following three cases.

\q1. \textit{Two arcs of the deleted edge~$\tilde e_1$ belong to different faces.} 

\noindent In this case we obtain a marked map on a sphere with $N-1$ edges and two faces. Let the first face contains $i$ arcs (Clearly,  
$1\le i\le 2N-3$). By lemma~\ref{l1_1} each of $\tilde{\varepsilon}_0(2N-2,i,2)$ such maps is obtained  exactly  $(i+1)(i+2)$ times. Thus, case \textbf{1} occurs for $\sum_{i=1}^{2N-3}{(i+1)(i+2)\tilde{\varepsilon}_0(2N-2,i,2)}$ initial maps.

\q2. \textit{Two arcs of the deleted edge are successive arcs of the face number~$1$.} 

\noindent
In this case we obtain a marked map on a sphere with $N-1$ edges and three faces. By lemma~\ref{l1_2} 
each of  $\varepsilon_0(N-1,3)$ such maps is obtained exactly twice. Thus, case \textbf{2} occurs for $2\varepsilon_0(N-1,3)$ initial maps.

\q3. \textit{Two arcs of the deleted edge are unsuccessive arcs of the face number~$1$.}

\noindent
Since $g=0$, the graph becomes disconnected, i.e.\ the case we consider corresponds to the case \textbf{4} of the description of the operation of deleting an edge. In this case we obtain an ordered pair of marked maps on a sphere, that contain together  $N-1$ edges and four faces. Such pair of maps can be of two types: one map of the pair contains one face and the other map contains three faces or both maps contain two faces. Consider these two subcases in details.

\q{3.1}. \textit{One map of the pair  contains one face, the other map contains three faces.} 

\noindent Let the map with one face contains  $i$ edges. By remark~\ref{rem:zero} the map with three faces contains at least 2 edges, hence, $1\le i\le N-3$. By lemma~\ref{l1_4} each of $2\varepsilon_0(i)\varepsilon_0(N-i-1,3)$ such pairs of maps occurs once. Hence, case 3.1 occurs for $2\sum_{i=1}^{N-3}{\varepsilon_0(i)\varepsilon_0(N-i-1,3)}$ initial maps.

\q{3.2}. \textit{Both maps of the pair contain two faces.} 

\noindent Let the first map of this pair contains $i$ edges. Clearly $1\le i\le N-2$. 
By lemma~\ref{l1_4} each of $\varepsilon_0(i,2)\varepsilon_0(N-i-1,2)$ such pairs of maps occurs twice. Hence, case 3.2 occurs for  $2\sum_{i=1}^{N-2}{\varepsilon_0(i,2)\varepsilon_0(N-i-1,2)}$ initial maps.

Summing these results and taking into account that $\varepsilon _0(0)=1$ and
$\varepsilon _0(0,2)=\varepsilon _0(1,3)=0$, we obtain the desired formula.
\end{proof}

\subsection{Explicit formula}

Before deducing an explicit formula for the numbers of gluing together three polygons into a sphere we prove some auxiliary lemmas.

\begin{lemma}
For any $N\geq 1$ the following equality holds: 
\begin{equation}
\sum_{i=1}^{2N-1}{\tilde{\varepsilon}_0(2N,i,2)}=\varepsilon_0(N,2)=N 4^{N-1}.
\label{eq:sum_eps}
\end{equation} 
\end{lemma}

\begin{proof}
This statement follows immediately from the definitions of $\tilde{\varepsilon}_0(2N,i,2)$ and $\varepsilon_0(N,2)$ and from formula~\eqref{eq:eps_0_2}.
\end{proof} 

\begin{lemma}
\label{lem:sum_i_eps}
For any $N\geq 1$ the following equality holds: 
\begin{equation}
\sum_{i=1}^{2N-1}{i\tilde{\varepsilon}_0(2N,i,2)}=N^2 4^{N-1}.
\label{eq:sum_i_eps}
\end{equation} 
\end{lemma}

\begin{proof}
Let $T(N)=\sum_{i=1}^{2N-1}{i\tilde{\varepsilon}_0(2N,i,2)}$.  We prove by induction that $T(N)=N^2 4^{N-1}$.

The base for $N=1$ is obvious. Let's prove the induction step. Let $T(k)=k^2 4^{k-1}$ for each $k<N$. 
We will show, that then $T(N)=N^2 4^{N-1}$. By formula \eqref{eq:eps_N_M_rec} we have:
\begin{align*}
T(N)=\sum_{i=1}^{2N-1}{i\tilde{\varepsilon}_0(2N,i,2)}=&
2\sum_{i=1}^{2N-1}{i\sum_{j=0}^{\left\lfloor\frac{i-3}{2}\right\rfloor}
{\varepsilon_0(j)\tilde{\varepsilon}_0(2N-2j-2,i-2j-2,2)}}+\\
&+\sum_{i=1}^{2N-1}{i(2N-i)\varepsilon_0(N-1)}=T_1(N)+T_2(N),
\end{align*}
where $T_1(N)$ and $T_2(N)$ are the first and the second summands of the obtained expression, respectively.
 
At first we calculate $T_1(N)$. Note, that
\begin{align*}
& T_1(N) =2\sum_{j=0}^{N-2}{\varepsilon_0(j)\sum_{i=2j+3}^{2N-1}{i\tilde{\varepsilon}_0(2N-2j-2,i-2j-2,2)}}.
\end{align*}

Let  $k=i-2j-2$  and $N_j=N-j-1$. 
By the induction assumption and formula~\eqref{eq:sum_eps} we have:
\begin{align*}
 T_1(N)&=2\sum_{j=0}^{N-2}{\varepsilon_0(j)\sum_{k=1}^{2N_j-1}{(k+2j+2)\tilde{\varepsilon}_0(2N_j,k,2)}}=\\
& =2\sum_{j=0}^{N-2}{\varepsilon_0(j)\Bigl(\sum_{k=1}^{2N_j-1}{k\tilde{\varepsilon}_0(2N_j,k,2)}+
(2j+2)\sum_{k=1}^{2N_j-1}{\tilde{\varepsilon}_0(2N_j,k,2)}\Bigr)}=\\
& =2\sum_{j=0}^{N-2}{\varepsilon_0(j)\left(N_j^2 4^{N_j-1}+(2j+2)N_j 4^{N_j-1}\right)}=\\
&=2^{2N-3}\sum_{j=0}^{N-2}{\frac{\varepsilon_0(j)}{4^j}(N^2-(j+1)^2)}.
\end{align*}

Applying  formulas~\eqref{eq:eps_B_0}, \eqref{eq:A_0_1} and formula \eqref{eq:A_k_k} for  $k=0,1$, we have:
\begin{align*}
T_1(N)&=2^{2N-3}\sum_{j=0}^{N-2}{\frac{(2j)!}{4^j j!(j+1)!}(N^2-(j+1)^2)}=\\
& =2^{2N-3}\biggl(N^2\sum_{j=0}^{N-2}{\frac{(2j)!}{4^j j!(j+1)!}}-\sum_{j=0}^{N-2}{\frac{(2j)!}{4^j j!j!}}-
\sum_{j=1}^{N-2}{\frac{(2j)!}{4^j j!(j-1)!}}\biggr)=\\
& =2^{2N-3}\left(N^2 \left(2-\frac{2N(2N-2)!}{4^{N-1}(N-1)!N!}\right)-
\frac{2(N-1)(2N-2)!}{4^{N-1}(N-1)!(N-1)!}-\right.\\
&\left.\phantom{=2^{2N-3}\left(\ \right.}-\frac{2}{3}\cdot\frac{(N-1)(N-2)(2N-2)!}{4^{N-1}(N-1)!(N-1)!}\right)=\\
&=N^2 4^{N-1}-\varepsilon_0(N-1)\frac{N(4N^2-1)}{3}.
\end{align*}
 
Now let's calculate $T_2(N)$:
$$
T_2(N)=\sum_{i=1}^{2N-1}{i(2N-i)\varepsilon_0(N-1)}=\varepsilon_0(N-1)\frac{N(4N^2-1)}{3}.
$$

Summing the expressions for $T_1(N)$ and $T_2(N)$ we obtain the desired formula.
\end{proof}

\begin{lemma}
For any $N\geq 1$ the following equality holds: 
\begin{equation}
\sum_{i=1}^{2N-1}{(i+1)(i+2)\tilde{\varepsilon}_0(2N,i,2)}=4^{N-2}N(N+1)(5N+7).
\label{eq:sum_i_i_eps}
\end{equation} 
\end{lemma}

\begin{proof}
Denote the left part of the formula~\eqref{eq:sum_i_i_eps} by $P(N)$ and the right part by $\tilde{P}(N)$. We will prove by induction that $P(N)=\tilde{P}(N)$.

The base for $N=1$  can be immediately verified. Let's prove the induction step.
The induction assumption is that $P(k)=\tilde{P}(k)$ for all $k<N$.   Let's show that then $P(N)=\tilde{P}(N)$.
By formula \eqref{eq:eps_N_M_rec} we have:
\begin{align*}
P(N)& =\sum_{i=1}^{2N-1}{(i+1)(i+2)\tilde{\varepsilon}_0(2N,i,2)}=\\
& =\sum_{i=1}^{2N-1}{(i+1)(i+2)\cdot 2\sum_{j=0}^{\left\lfloor\frac{i-3}{2}\right\rfloor}
{\varepsilon_0(j)\tilde{\varepsilon}_0(2N-2j-2,i-2j-2,2)}}+\\
&\phantom{=\ } +\sum_{i=1}^{2N-1}{(i+1)(i+2)(2N-i)\varepsilon_0(N-1)}=P_1(N)+P_2(N),
\end{align*}
where  $P_1(N)$ and $P_2(N)$ are the first and  the second summands of the obtained expression, respectively.

At first we calculate $P_1(N)$. Note, that
\begin{equation*}
P_1(N)=2\sum_{j=0}^{N-2}{\varepsilon_0(j)\biggl(\sum_{i=2j+3}^{2N-1}{(i+1)(i+2)\tilde{\varepsilon}_0(2N-2j-2,i-2j-2,2)}\biggr)}.
\end{equation*}
As in lemma \ref{lem:sum_i_eps} we set $k=i-2j-2$; $N_j=N-j-1$. 
Then
\begin{align*}
P_1(N)&=2\sum_{j=0}^{N-2}{\varepsilon_0(j)\biggl(\sum_{k=1}^{2N_j-1}{(k+1+2j+2)(k+2+2j+2)\tilde{\varepsilon}_0(2N_j,k,2)}\biggr)}=\\
& =2\sum_{j=0}^{N-2}{\varepsilon_0(j)\sum_{k=1}^{2N_j-1}{(k+1)(k+2)\tilde{\varepsilon}_0(2N_j,k,2)}}+\\
&\phantom{=\ } +2\cdot 4\sum_{j=0}^{N-2}{\varepsilon_0(j)(j+1)\sum_{k=1}^{2N_j-1}{k\tilde{\varepsilon}_0(2N_j,k,2)}}+\\
&\phantom{=\ } +2\cdot 2\sum_{j=0}^{N-2}{\varepsilon_0(j)(j+1)(2j+5)\sum_{k=1}^{2N_j-1}{\tilde{\varepsilon}_0(2N_j,k,2)}}.
\end{align*} 
 
Applying induction assumption and formulas \eqref{eq:sum_i_eps} and \eqref{eq:sum_eps} we have:
\begin{align*}
& P_1(N)=2\sum_{j=0}^{N-2}{\varepsilon_0(j)4^{N_j-2}R_j(N)},
\end{align*} 
where
\begin{eqnarray*}
R_j(N)&=&N_j(N_j+1)(5N_j+7)+16(j+1)N_j^2+8(j+1)(2j+5)N_j=\\
&=&N(N+1)(5N+7)+(N^2-N-24)(j+1)-\\
&&-(N+27)(j+1)j-5(j+1)j(j-1).
\end{eqnarray*}
 
Applying formulas \eqref{eq:eps_B_0}, \eqref{eq:A_0_1} and formula \eqref{eq:A_k_k} for $k=0,1,2$ we have:
\begin{align*}
 P_1(N)&=2\sum_{j=0}^{N-2}{\varepsilon_0(j)4^{N-j-3}R_j(N)}=
2^{2N-5}\sum_{j=0}^{N-2}{\frac{(2j)!}{4^j j!(j+1)!}R_j(N)}=\\
& =2^{2N-5}N(N+1)(5N+7)\sum_{j=0}^{N-2}{\frac{(2j)!}{4^j j!(j+1)!}}+\\
&\phantom{=\ } +2^{2N-5}(N^2-N-24)\sum_{j=0}^{N-2}{\frac{(2j)!}{4^j j!j!}}-\\
&\phantom{=\ } -2^{2N-5}(N+27)\sum_{j=1}^{N-2}{\frac{(2j)!}{4^j j!(j-1)!}}-
2^{2N-5}\cdot 5\sum_{j=2}^{N-2}{\frac{(2j)!}{4^j j!(j-2)!}}=\\
& =2^{2N-5}N(N+1)(5N+7)\left(2-\frac{2N(2N-2)!}{4^{N-1}(N-1)!N!}\right)+\\
&\phantom{=\ } +2^{2N-5}(N^2-N-24)\frac{2(N-1)(2N-2)!}{4^{N-1}(N-1)!(N-1)!}-\\
&\phantom{=\ } -2^{2N-5}(N+27)\frac{2(N-1)(N-2)(2N-2)!}{3\cdot 4^{N-1}(N-1)!(N-1)!}-\\
&\phantom{=\ } -2^{2N-5}\frac{5\cdot 2(N-1)(N-2)(N-3)(2N-2)!}{5 \cdot4^{N-1}(N-1)!(N-1)!}=\\
& =4^{N-2}N(N+1)(5N+7)-\frac{(2N-2)!}{4(N-1)!N!}R(N),
\end{align*}
where
\begin{multline*}
 R(N)=N^2(N+1)(5N+7)-(N^2-N-24)(N-1)N+\\
 +\frac{1}{3}(N+27)(N-2)(N-1)N+(N-3)(N-2)(N-1)N=\\
 =4N\left(\frac{4}{3}N^3+4N^2+\frac{11}{3}N-3\right).
\end{multline*}

By formula \eqref{eq:eps_B_0} we finally obtain:
\begin{align*}
& P_1(N)=4^{N-2}N(N+1)(5N+7)-\varepsilon_0(N-1)N\left(\frac{4}{3}N^3+4N^2+\frac{11}{3}N-3\right).
\end{align*}
 
Now let's calculate $P_2(N)$.
\begin{align*}
 P_2(N)&=\sum_{i=1}^{2N-1}{(i+1)(i+2)(2N-i)\varepsilon_0(N-1)}=\\
& =\varepsilon_0(N-1)\sum_{i=1}^{2N-1}\left(4N+(6N-2)i+(2N-3)i^2-i^3\right)=\\
& =\varepsilon_0(N-1)N\left(\frac{4}{3}N^3+4N^2+\frac{11}{3}N-3\right).
\end{align*} 
 
Summing the expressions for $P_1(N)$ and $P_2(N)$ we obtain the desired formula.
\end{proof}

\begin{theorem}\label{th:eps_0_3}
For $N\geq 1$ the numbers $\varepsilon_0(N,3)$ satisfy  the following equality:
\begin{equation}
\label{eq:eps_0_3}
\varepsilon_0(N,3)=\frac{(8N+5)(N-1)N(N+1)}{210}C_{2N+1}^{N}.
\end{equation}
\end{theorem}

\begin{proof}
Let $\tilde{C}(N)=\frac{(8N+5)(N-1)N(N+1)}{210}C_{2N+1}^{N}$. We  will prove that $\varepsilon_0(N,3)=\tilde{C}(N)$ by induction.

The base for $N=1$  is clear. For $N=2$ it follows from the formulas \eqref{eq:eps_0_3_rec} and \eqref{eq:eps_N_M_rec}  that
 $\varepsilon_0(2,3)=6=\tilde{C}(2)$. 

Let's prove the induction step. Let $\varepsilon_0(k,3)=\tilde{C}(k)$ for  $k\leq N$. We will show, that then $\varepsilon_0(N+1,3)=\tilde{C}(N+1)$. By formula \eqref{eq:eps_0_3_rec} we have:
\begin{align*}
 \varepsilon_0(N+1,3)&=2\sum_{i=0}^{N-1}{\varepsilon_0(i)\varepsilon_0(N-i,3)}+ 
2\sum_{i=0}^{N-1}{\varepsilon_0(i,2)\varepsilon_0(N-i,2)}+ \\
& +\sum_{i=1}^{2N-1}{(i+1)(i+2)\tilde{\varepsilon}_0(2N,i,2)}=S_1(N)+S_2(N)+S_3(N),
\end{align*} 
where $S_1(N)$, $S_2(N)$ and $S_3(N)$ are the first, the second and the third summands of the obtained expression, respectively.
 
By formula \eqref{eq:sum_i_i_eps} we conclude:
\begin{equation}
\label{eq:S3}
S_3(N)=\sum_{i=1}^{2N-1}{(i+1)(i+2)\tilde{\varepsilon}_0(2N,i,2)}=4^{N-2}N(N+1)(5N+7).
\end{equation}

Applying formula \eqref{eq:eps_0_2} we obtain:
\begin{align}
\label{eq:S2}
& S_2(N)=2\sum_{i=0}^{N-1}{\varepsilon_0(i,2)\varepsilon_0(N-i,2)}=
2\sum_{i=1}^{N-1}{i4^{i-1}\cdot (N-i)4^{N-i-1}}=\notag\\
& =2\cdot 4^{N-2}\left(\frac{N^2(N-1)}{2}-\frac{(N-1)N(2N-1)}{6}\right)=\frac{4^{N-2}}{3}(N-1)N(N+1).
\end{align}

Now consider $S_1(N)$. Since  $\varepsilon_0(N-i,3)=0$ for $i=N$, we have:
\begin{displaymath}
S_1(N)=2\sum_{i=0}^{N-1}{\varepsilon_0(i)\varepsilon_0(N-i,3)}=2\sum_{i=0}^{N}{\varepsilon_0(i)\varepsilon_0(N-i,3)}.
\end{displaymath}

Applying the induction assumption for $\varepsilon_0(N-i,3)$ and formula \eqref{eq:eps_B_0} we obtain:
\begin{align*}
 S_1(N)&=\sum_{i=0}^{N}{\frac{2(2i)!(8(N-i)+5)(N-i-1)(N-i)(N-i+1)}{i!(i+1)!\cdot 210}C_{2(N-i)+1}^{N-i}}=\\
& =\frac{1}{105}\sum_{i=0}^{N}{R_i(N)\frac{C_{2i+1}^{i}\cdot C_{2N-2i+1}^{N-i}}{2i+1}},
\end{align*} 
where
\begin{align*}
 R_i(N)&=(8(N-i)+5)(N-i-1)(N-i)(N-i+1)=\\
& =N(N+1)(N+2)(8N+13)-(i+1)N(N+1)(32N+31)+\\
&\phantom{=\ }+i(i+1)N(48N+15)-(i-1)i(i+1)(32N-11)+\\
&\phantom{=\ }+8(i-2)(i-1)i(i+1).
\end{align*} 
 
Applying formula \eqref{eq:D_k_N}, and, after that, formulas (\ref{eq:D_0}-\ref{eq:D_4}), we can write the following:
\begin{align*}
105 S_1(N)&=N(N+1)(N+2)(8N+13)D_0(N)-\\
&\phantom{=\ }-N(N+1)(32N+31)D_1(N)+N(48N+15)D_2(N)-\\
&\phantom{=\ }-(32N-11)D_3(N)+8D_4(N)=\\
& =N(N+1)(N+2)(8N+13)C_{2N+2}^{N}-\\
&\phantom{=\ }-N(N+1)(32N+31)\left(2\cdot 4^N-\frac{1}{2}C_{2N+2}^{N+1}\right)+\\
&\phantom{=\ }+N(48N+15)(N+1)\left(4^N-\frac{1}{2}C_{2N+2}^{N+1}\right)-\\
&\phantom{=\ }-(32N-11)N(N+1)\left(3\cdot 4^{N-1}-\frac{1}{2}C_{2N+2}^{N+1}\right)+\\
&\phantom{=\ }+8(N-1)N(N+1)\left(10\cdot 4^{N-2}-\frac{1}{2}C_{2N+2}^{N+1}\right).
\end{align*} 

Let's substitute in the first summand $C_{2N+2}^{N}$ by the following expression:
\begin{align*}
C_{2N+2}^{N}=\frac{1}{2}C_{2N+3}^{N+1}-\frac{1}{2}\cdot\frac{C_{2N+2}^{N+1}}{N+2}.
\end{align*} 
After this substitution all summands having the multiple $C_{2N+2}^{N+1}$ cancel on.
Transforming other summands we get the following formula:
\begin{align}
\label{eq:S1}
& S_1(N)=\frac{1}{210}N(N+1)(N+2)(8N+13)C_{2N+3}^{N+1}-\frac{4^{N-1}}{3}N(N+1)(4N+5).
\end{align}

Summing the expressions for $S_1(N)$, $S_2(N)$ and $S_3(N)$ (formulas \eqref{eq:S1}, \eqref{eq:S2}
and \eqref{eq:S3}) we obtain the following:
\begin{align*}
 \varepsilon_0(N+1,3)&=S_1(N)+S_2(N)+S_3(N)=\\
& =\frac{1}{210}N(N+1)(N+2)(8N+13)C_{2N+3}^{N+1}=\tilde{C}(N+1).
\end{align*}
\end{proof}

\section{Gluing together two polygons into a torus}

In this section we prove recurrence and explicit formulas for numbers $\varepsilon_1(N,2)$. 
Remind, that  $\varepsilon_1(N,2)$ is the number of marked maps on a torus (surface of genus 1), that contain $N$ edges and 2 faces.

\begin{theorem}
For $N\geq 3$ the numbers $\varepsilon_1(N,2)$ satisfy the following recursion:
\begin{multline}
\varepsilon_1(N,2)=2\sum_{i=0}^{N-3}{\Bigl(\varepsilon_0(i)\varepsilon_1(N-i-1,2)+
\varepsilon_0(i,2)\varepsilon_1(N-i-1)\Bigr)}+\\
+N(2N-1)\varepsilon_1(N-1)+\varepsilon_0(N-1,3).
\label{eq:eps_1_2_rec}
\end{multline}  
\end{theorem}

\begin{proof} Consider all marked maps on a torus, that contain $N>2$ edges and 2 faces and apply to each of them the operation of deleting an edge. For any map we have one of the following four cases.

\smallskip
\q1. \textit{Two arcs of the deleted edge~$\tilde e_1$ belong to different faces.} 

\noindent In this case we obtain a map on a torus with  $N-1$ edges and one face.  By lemma~\ref{l1_1} each of $\varepsilon_1(N-1)$ such maps is obtained $N(2N-1)$ times. Thus, case \textbf{1} occurs for  $N(2N-1)\varepsilon_1(N-1)$ initial maps.

\q2. \textit{Two arcs of the deleted edge are successive arcs of the face number~$1$.} 

\noindent
In this case we obtain a marked map on a torus with $N-1$ edges and two faces. By lemma~\ref{l1_2} each of $\varepsilon_1(N-1,2)$ such maps occurs twice. Thus, case~\textbf{2} occurs for  $2\varepsilon_1(N-1,2)$ initial maps.

\q3. \textit{Two arcs of the deleted edge are unsuccessive arcs of the face number~$1$. The obtained graph is connected.} 

\noindent In this case we obtain a marked map on a sphere with  $N-1$ edges and three faces. By lemma~\ref{l1_3} each of $\varepsilon_0(N-1,3)$ such maps occurs once. Thus, case \textbf{3} occurs for $\varepsilon_0(N-1,3)$ initial maps.

\q4. \textit{Two arcs of the deleted edge are unsuccessive arcs of the face number~$1$. The obtained graph is disconnected.} 

\noindent  In this case we obtain an ordered pair of maps, that contain together $N-1$ edges and three faces. By lemma~\ref{l1_4} each such pair occurs once. The maps in such pair   have sum of the genera~1, i.e.\ one of them is drawn on a torus and the other --- on a sphere. Moreover, one of maps contains one face and the other --- two faces. Thus we have two subcases:  the map of genus 0 in our pair has one face or it has two faces. Let's count the number of pairs for each subcase.

\q{4.1}. \textit{The map of genus 0  has one face.}  Let this map has $i$ edges. By remark~\ref{rem:zero} the other map contains at least 3 edges, consequently, $1\le i\le N-4$. Clearly, the number of such ordered pairs is $2\varepsilon_0(i)\varepsilon_1(N-i-1,2)$. Thus, case~\textbf{4.1} occurs for $2\sum_{i=1}^{N-4}{\varepsilon_0(i)\varepsilon_1(N-i-1,2)}$  initial maps.

\q{4.2}. \textit{The map of genus 0  has two faces.}  Let this map  has $i$ edges. By remark~\ref{rem:zero} the other map contains at least 2 edges, consequently, $1\le i\le N-3$.
Clearly, the number of such ordered pairs is $2\varepsilon_0(i,2)\varepsilon_1(N-i-1)$.  Thus, case~\textbf{4.2} occurs for $2\sum_{i=1}^{N-3}{\varepsilon_0(i,2)\varepsilon_1(N-i-1)}$ initial maps.

Summing the formulas for all these cases and taking into account that $\varepsilon _0(0)=1$ and $\varepsilon _0(0,2)=\varepsilon _1(2,2)=0$, 
we obtain the desired formula.
\end{proof}

\begin{theorem}
For $N\geq 1$ the numbers  $\varepsilon_1(N,2)$ satisfy the following formula:
\begin{equation}
\label{eq:eps_1_2}
\varepsilon_1(N,2)=4^{N-4}\frac{N(N-1)(N-2)(13N+3)}{3}.
\end{equation}
\end{theorem}

\begin{proof}
Let $F(N)=4^{N-4}\frac{N(N-1)(N-2)(13N+3)}{3}$. We will prove the statement $\varepsilon_1(N,2)=F(N)$ by induction. 

The base for $N=1,2$ immediately follows from remark~\ref{rem:zero}, since $\varepsilon_1(1,2)=F(1)=\varepsilon_1(2,2)=F(2)=0$. Let's prove the induction step. Let for all $k<N$ we have $\varepsilon_1(k,2)=F(k)$. Then we will show that $\varepsilon_1(N,2)=F(N)$. By formula \eqref{eq:eps_1_2_rec} we have:
\begin{align*}
 \varepsilon_1(N,2)&=2\sum_{i=0}^{N-3}{\varepsilon_0(i)\varepsilon_1(N-i-1,2)}+
2\sum_{i=0}^{N-3}{\varepsilon_0(i,2)\varepsilon_1(N-i-1)}+\\
& +\Bigl(N(2N-1)\varepsilon_1(N-1)+\varepsilon_0(N-1,3)\Bigr)=S_1(N)+S_2(N)+S_3(N),
\end{align*}
where $S_1(N)$, $S_2(N)$, $S_3(N)$~--- are the first, the second and the third summands of the obtained expression respectively.

At first let's calculate $S_3(N)$. By formulas \eqref{eq:eps_1} and \eqref{eq:eps_0_3} we have:
\begin{align}
\label{eq:S3_tor}
 S_3(N)&=N(2N-1)\varepsilon_1(N-1)+\varepsilon_0(N-1,3)=\notag\\
& =\frac{N(2N-1)(2N-2)!}{12(N-1)!(N-3)!}+\frac{(8(N-1)+5)(N-2)(N-1)N}{210}C_{2N-1}^{N-1}=\notag\\
& =\frac{(51N-6)(2N-1)!}{420(N-1)!(N-3)!}.
\end{align}

Now, taking into account $\varepsilon_0(0,2)=0$, let's modify the sum~$S_2(N)$. Let~$j=N-1-i$. By formulas \eqref{eq:eps_0_2}, \eqref{eq:eps_1} and formula~\eqref{eq:A_k_k} for $k=2,3$ we have:

\begin{align}
\label{eq:S2_tor}
& S_2(N)=2\sum_{i=1}^{N-3}{\varepsilon_0(i,2)\varepsilon_1(N-i-1)}=2\sum_{j=2}^{N-2}{\varepsilon_0(N-1-j,2)\varepsilon_1(j)}=\notag\\
& =2\sum_{j=2}^{N-2}{(N-1-j)4^{N-1-j-1}\frac{(2j)!}{12j!(j-2)!}}=\notag\\
& =\frac{2^{2N-5}}{3}\left((N-3)\sum_{j=2}^{N-2}{\frac{(2j)!}{4^j j!(j-2)!}}-
\sum_{j=3}^{N-2}{\frac{(2j)!}{4^j j!(j-3)!}}\right)=\notag\\
& =\frac{2^{2N-5}}{3}\left((N-3)\frac{2(N-3)(N-2)(N-1)(2N-2)!}{5\cdot 4^{N-1}(N-1)!(N-1)!}-\right.\notag\\
& \left.\phantom{=\frac{2^{2N-5}}{3}\left(\right.}-\frac{2(N-4)(N-3)(N-2)(N-1)(2N-2)!}{7\cdot 4^{N-1}(N-1)!(N-1)!}\right)=\notag\\
& =\frac{(N-3)(2N-1)!}{420(N-1)!(N-3)!}.
\end{align}

Now let's modify $S_1(N)$. By the induction assumption:
\begin{multline*}
 \varepsilon_1(N-1-i,2)=\\
=(N-i-1)(N-i-2)(N-i-3)(13N-13i-10)\frac{4^{N-i-5}}{3}=\\
=s_i(N)\frac{2^{2N-10}}{3\cdot 4^i},
\end{multline*}
where
\begin{align*}
 s_i(N)&=(N-i-1)(N-i-2)(N-i-3)(13N-13i-10)=\\
& =N(N-1)(N-2)(13N+3)-(N-1)(N-2)(52N-30)(i+1)+\\
&\phantom{=\ }+3(N-2)(26N-36)i(i+1)-(52N-114)(i-1)i(i+1)+\\
&\phantom{=\ }+13(i-2)(i-1)i(i+1).
\end{align*}

Taking into account the notation defined above, formula~\eqref{eq:eps_B_0} and that $\varepsilon_1(N-1-i,2)=0$ for $i=N-3$, we obtain:
\begin{align*}
& S_1(N)=2\sum_{i=0}^{N-4}{\varepsilon_0(i)\varepsilon_1(N-i-1,2)}=
\frac{2^{2N-9}}{3}\sum_{i=0}^{N-4}{\frac{(2i)!}{4^i i!(i+1)!}s_i(N)}. 
\end{align*}

By the formula for~$s_i(N)$, formula~\eqref{eq:A_0_1} and formula~\eqref{eq:A_k_k} for $k=0,\ldots,3$ we have:
\begin{align*}
 \frac{3S_1(N)}{2^{2N-9}}&=N(N-1)(N-2)(13N+3)\sum_{i=0}^{N-4}{\frac{(2i)!}{4^i i!(i+1)!}}-\\
&\phantom{=\ } -(N-1)(N-2)(52N-30)\sum_{i=0}^{N-4}{\frac{(2i)!}{4^i i!i!}}+\\
&\phantom{=\ } +3(N-2)(26N-36)\sum_{i=1}^{N-4}{\frac{(2i)!}{4^i i!(i-1)!}}-\\
&\phantom{=\ } -(52N-114)\sum_{i=2}^{N-4}{\frac{(2i)!}{4^i i!(i-2)!}}+
13\sum_{i=3}^{N-4}{\frac{(2i)!}{4^i i!(i-3)!}}=\\
& =N(N-1)(N-2)(13N+3)\left(2-\frac{2(N-2)(2N-6)!}{4^{N-3}(N-3)!(N-2)!}\right)-\\
&\phantom{=\ } -(N-1)(N-2)(52N-30)\frac{2(N-3)(2N-6)!}{4^{N-3}(N-3)!(N-3)!}+\\
&\phantom{=\ } +3(N-2)(26N-36)\frac{2(N-4)(N-3)(2N-6)!}{3\cdot 4^{N-3}(N-3)!(N-3)!}-\\
&\phantom{=\ } -(52N-114)\frac{2(N-5)(N-4)(N-3)(2N-6)!}{5\cdot 4^{N-3}(N-3)!(N-3)!}+\\
&\phantom{=\ } +13\frac{2(N-6)(N-5)(N-4)(N-3)(2N-6)!}{7\cdot 4^{N-3}(N-3)!(N-3)!}=\\
& =2N(N-1)(N-2)(13N+3)-\frac{2(N-2)(N-1)(2N-6)!}{35\cdot 4^{N-3}(N-3)!(N-1)!}R(N),
\end{align*} 
where
\begin{multline*}
 R(N)=35N(N-1)(N-2)(13N+3)+\\
 +35(N-1)(N-2)(N-3)(52N-30)-\\
 -35(N-2)(N-3)(N-4)(26N-36)+\\
 +7(N-3)(N-4)(N-5)(52N-114)-\\
 -5\cdot 13(N-3)(N-4)(N-5)(N-6)=\\
 =4(52N-9)(2N-5)(2N-3)(2N-1).
\end{multline*}
 
Whence it follows that
\begin{equation}
\label{eq:S1_tor}
 S_1(N) =\frac{4^{N-4}}{3}N(N-1)(N-2)(13N+3)-\frac{(52N-9)(2N-1)!}{420(N-3)!(N-1)!}.
\end{equation}
 
Summing the expressions for~$S_1(N)$, $S_2(N)$ and $S_3(N)$ (formulas \eqref{eq:S1_tor}, 
\eqref{eq:S2_tor} and \eqref{eq:S3_tor}), we obtain:
\begin{align*}
 \varepsilon_0(N,3)&=S_1(N)+S_2(N)+S_3(N)=\\
& =\frac{4^{N-4}}{3}N(N-1)(N-2)(13N+3)=F(N).
\end{align*}
\end{proof}

\end{document}